\documentclass[a4paper,10pt,spanish]{article}
\usepackage{amsthm}
\usepackage[ansinew]{inputenc}
\usepackage[english]{babel}
\usepackage{anysize}
\usepackage{amsfonts}
\usepackage{amssymb}
\usepackage{amsmath}
\usepackage{doc}
\usepackage{exscale}
\usepackage{fontenc}
\usepackage{ifthen}
\usepackage{latexsym}
\usepackage{makeidx}
\usepackage{syntonly}
\usepackage{graphicx}
\usepackage{float}
\usepackage{array}
\usepackage[all]{xy}
\usepackage[usenames,dvipsnames]{color}

\marginsize{2.5cm}{2.5cm}{2.5cm}{2.5cm}

\newtheorem{thm}{Theorem}[section]
\newtheorem{cor}[thm]{Corollary}
\newtheorem{lem}[thm]{Lemma}
\newtheorem{pro}[thm]{Proposition}
\newtheorem{df}[thm]{Definition}
\newtheorem{rmk}[thm]{Remark}
\newtheorem{que}[thm]{Question}
\newtheorem{ex}[thm]{Example}
\newtheorem{lem-df}[thm]{Lemma-Definition}
\newtheorem{conj}[thm]{Conjecture}
\newtheorem{rmk-df}[thm]{Remark-Definition}

\newcommand{\beq}{\begin{equation}}
\newcommand{\enq}{\end{equation}}
\newcommand{\beqn}{\begin{equation*}}
\newcommand{\enqn}{\end{equation*}}

\newcommand{\ra}{\rightarrow}

\newcommand{\Ra}{\Rightarrow}

\newcommand{\longra}{\longrightarrow}

\newcommand{\mC}{\mathbb{C}}

\newcommand{\mG}{\mathbb{G}}
\newcommand{\mH}{\mathbb{H}}

\newcommand{\mP}{\mathbb{P}}

\newcommand{\mQ}{\mathbb{Q}}

\newcommand{\caC}{\mathcal{C}}

\newcommand{\caF}{\mathcal{F}}
\newcommand{\caG}{\mathcal{G}}
\newcommand{\caH}{\mathcal{H}}

\newcommand{\caK}{\mathcal{K}}

\newcommand{\caO}{\mathcal{O}}

\DeclareMathOperator{\caExt}{{\mathcal Ext}}
\DeclareMathOperator{\caHom}{{\mathcal Hom}}

\DeclareMathOperator{\coker}{coker}
\DeclareMathOperator{\im}{im}

\DeclareMathOperator{\codim}{codim}
\DeclareMathOperator{\Pic}{Pic}

\DeclareMathOperator{\Sym}{Sym}

\DeclareMathOperator{\Supp}{Supp}
\DeclareMathOperator{\rk}{rk}

\title{Grassmannian BGG Complexes and Hodge numbers of irregular varieties}
\author{V\'{\i}ctor Gonz\'{a}lez-Alonso
\footnote{Previous address: Departament de Matem\`atica Aplicada I, Universitat Polit\`ecnica de Catalunya (BarcelonaTech), Av. Diagonal 647, 08028 Barcelona, Spain.}
\thanks{The author developed this work supported by grant FPU-AP2008-01849 of the Ministerio de Educaci\'{o}n, project 2009-SGR-1284 of the Generalitat de Catalunya, and projects MTM2009-14163-C02-02 and MTM2012-38122-C03-01/FEDER of the Ministerio de Econom{\'\i}a y Competitividad.}\\
\small{Instit\"ut fur Algebraische Geometrie, Leibniz Universit\"at Hannover}\\
\small{Welfengarten 1, 30167 Hannover, Germany} \\
\small{gonzalez@math.uni-hannover.de}
}
\date{\today}

\begin{document}

\maketitle

\begin{abstract}
In this paper we investigate the exactness of the Grassmannian BGG complexes introduced in \cite{VGA1}, and obtain some inequalities between some Hodge numbers of some irregular varieties. In particular, we obtain sharp lower bounds for the Hodge numbers of smooth subvarieties of Abelian varieties, as well as some improvements of results of Lazarsfeld and Popa \cite{LP} and Lombardi \cite{Lom} concerning threefolds and fourfolds.

\end{abstract}

\section{Introduction and preliminaries}
\label{sect:intro}

In the classification of higher dimensional algebraic varieties, it is interesting to give numerical conditions that imply the existence of some special geometric structure, as for example, a fibration over another variety of smaller dimension. A paradigmatical example is the classical Castelnuovo-de Franchis inequality, which says that if the geometric genus $p_g\left(S\right)$ and the irregularity $q\left(S\right)$ of an irregular surface $S$ satisfy
\beq \label{CdF-ineq-intro}
p_g\left(S\right) \leq 2q\left(S\right)-4,
\enq
then there exists a fibration $f: S \ra C$ over a smooth curve of genus $g\left(C\right) \geq 2$.

The Castelnuovo-de Franchis inequality (\ref{CdF-ineq-intro}) admits several generalizations to higher dimensions. On the one hand, the {\em Generalized Castelnuovo-de Franchis Theorem}, proved independently by Ran \cite{Ran} and Catanese \cite{Cat}, implies that any irregular variety $X$ without {\em higher irrational pencils}\footnote{Fibrations over a variety of Albanese general type (of maximal Albanese dimension but non-surjective Albanese map). They are higher dimensional analogues to fibrations over curves of genus $g \geq 2$.} satisfies
\beqn
h^{k,0}\left(X\right) > k\left(q\left(X\right)-k\right)
\enqn
for every $k = 1, \ldots, \dim X$. On the other hand, the study of the relations between derivative complexes, cohomological support loci $V^i\left(X,\omega_X\right)$ and irregular fibrations carried out by Green and Lazarsfeld in \cite{GL1,GL2} lead to the inequality
\beqn
\chi\left(X,\omega_X\right) \geq q\left(X\right)-\dim X
\enqn
for a variety $X$ with no higher irrational pencil. This inequality was first obtained by Pareschi and Popa \cite{PP} using the Fourier-Mukai transform, and later by Lazarsfeld and Popa \cite{LP} by means of the BGG complex of $X$. As a third generalization, in \cite{VGA1} we proved that a variety $X$ with no higher irrational pencil and irregularity $q\left(X\right) \geq 2\dim X$ satisfies
\beq \label{VGA1-ineq-intro}
h^{2,0}\left(X\right) \geq 2 \left(\dim X - 1\right)q\left(X\right) - \binom{2\dim X - 1}{2},
\enq
while if $q \left(X\right) < 2 \dim X$ we recovered the bound $h^{2,0}\left(X\right) \geq \binom{q\left(X\right)}{2}$ obtained by Causin and Pirola \cite{CP}. The main tool used in \cite{VGA1} is the {\em Grassmannian BGG complex}, made up by glueing the {\em higher rank derivative complexes} (cf. Section \ref{sect:higher}).

In this paper we slightly generalize these complexes, and prove some general results about their exactness (based on an idea of Green and Lazarsfeld in \cite{GL1}). As a consequence, we obtain a bunch of inequalities between the Hodge numbers of varieties admitting {\em non-degenerate} subspaces of holomorphic 1-forms. Then in Section \ref{sect:subvAV} we consider the special case of smooth subvarieties of Abelian varieties (or more generally, \'etale covers of them), provinding examples that show that most of the new inequalities are sharp. Finally, in Section \ref{sect:h20} we apply some ideas of Lazarsfeld and Popa \cite{LP} and Lombardi \cite{Lom} to our complexes, obtaining inequalities for the Hodge numbers. In this way we improve some inequalities of Lombardi \cite{Lom} for threefolds and fourfolds, and also recover (\ref{VGA1-ineq-intro}) with slightly stronger hypothesis than in \cite{VGA1}.

{\bf Some notation and definitions:} Through all the paper, $X$ will denote a complex smooth irregular projective (or more generally, compact K\"ahler) variety of dimension $d = \dim X$ and irregularity $q = q\left(X\right)$. Quite often, for the sake of brevity, we will denote by $V = H^0\left(X,\Omega_X^1\right)$ the space of holomorphic 1-forms on $X$.

If $E$ is a vector space (or a vector bundle over some variety), we will denote by $\Sym^r E$ its $r$-th symmetric power, which we consider both as a quotient and as a subspace of $E^{\otimes r}$. Its elements will be written with multiplicative notation, denoting by $e_1\cdots e_r$ and $e_1^r$ the classes of $e_1 \otimes \cdots \otimes e_r$ or $e_1^{\otimes r}$ respectively ($e_i \in E$).

We will denote by $\mG_k = Gr\left(k,V\right)$ the Grassmannian of $k$-dimensional subspaces of $V$, by $S \subset \mG_k \times V$ the tautological subbundle of $\mG_k$, and by $Q=\left(\mG_k \times V\right)/S$ its tautological quotient bundle. For some explicit computations in the cohomology algebra of $\mG_k$, we will use the usual notation for Schubert classes: if $\lambda=(q-k\geq\lambda_1\geq\lambda_2\geq\cdots\geq\lambda_k\geq0)$ is a partition, $\sigma_{\lambda}$ will be the cohomology class of the Schubert cycle
\beqn
\Sigma_{\lambda} = \left\{W\in\mG_k \,\vert\,\dim\left(W\cap\mC\left\langle v_1,\ldots,v_{q-k+i-\lambda_i}\right\rangle\right)\geq i \right\},
\enqn
which is independent of the basis $\left\{v_1,\ldots,v_q\right\}$ of $V$ chosen.

Finally, we will often use the following definition for complexes of vector spaces.
\begin{df}
We say that a complex of vector spaces
$$0 \longra V_0 \stackrel{\phi_0}{\longra} V_1 \stackrel{\phi_1}{\longra} \cdots \longra V_k \stackrel{\phi_k}{\longra} \cdots$$
is {\em exact in the first $n$ steps} if the truncated complex
$$0 \longra V_0 \longra V_1 \longra \cdots \longra V_n$$
is exact, or equivalently, if the (co)homology groups $H^i = \ker \phi_i / \im \phi_{i-1}$ vanish for $i < n$.
\end{df}

{\bf Acknowledgements:} I would like to thank Gian Pietro Pirola and Mihnea Popa for the many useful discussions and their hospitality during the stays that led to this work. I would also like to thank my PhD advisors, Miguel \'Angel Barja and Juan Carlos Naranjo, for their friendship, constant support and advice.


\section{Higher rank derivative complexes}
\label{sect:higher}

In this section we present the main definitions and most general results of the paper. We first introduce our main tools, which we call {\em higher rank derivative complex} and {\em Grassmannian BGG complex}, and which are slightly more general than the version we defined in \cite{VGA1}. The reason of these names is that they generalize the derivative and BGG complexes, respectively, to the case where more than one 1-form (or cohomology class $v \in H^1\left(X,\caO_X\right)$) are put into the picture. However, we do not obtain them from a ``derivative'' setting, nor from a categorical analogue to the BGG correspondence. Instead, we construct them directly and show that they coincide with the previous ones in the case of one-dimensional subspaces.

\begin{df}[Higher-rank derivative complex] \label{df-complexes}
Fix integers $r \geq 1$, $1 \leq n \leq \min\left\{r,d\right\}$, $0 \leq j \leq d$, and a linear subspace $W \subseteq V = H^0\left(X,\Omega_X^1\right)$. We define $C^j_{r,n,W}$ as the complex (of vector spaces)
\begin{multline} \label{Complex-W}
0 \longra \Sym^r W \otimes H^j\left(X,\caO_X\right) \longra \Sym^{r-1} W \otimes H^j\left(X,\Omega_X^1\right)\longra \cdots \\
\cdots \longra \Sym^{r-i} W \otimes H^j\left(X,\Omega_X^i\right) \longra \cdots \\
\cdots \longra \Sym^{r-n} W \otimes H^j\left(X,\Omega_X^n\right)
\end{multline}
where the maps $\mu_i^j: \Sym^{r-i}W \otimes H^j\left(X,\Omega_X^i\right) \ra \Sym^{r-i-1}W \otimes H^j\left(X,\Omega_X^{i+1}\right)$ are given by
\beqn
\mu_i^j\left(\left(w_1 \cdots w_{r-i}\right)\otimes[\alpha]\right) = \sum_{t=1}^{r-i} \left(w_1 \cdots \widehat{w_t} \cdots w_{r-i}\right)\otimes\left[w_t \wedge \alpha\right].
\enqn
\end{df}

It is immedate to check that the maps $\mu_i^j$ are well defined and indeed make $C^j_{r,n,W}$ into a complex. Since for every $1 \leq n' < n$ the complex $C^j_{r,n',W}$ is a truncation of $C^j_{r,n,W}$, we may assume that $n$ is always the greatest possible, that is, $n = \min\{r,d\}$, and denote the complex simply by $C^j_{r,W}$. Note that in the case of a 1-dimensional $W$, generated by $w$, we have $\Sym^r W \equiv \mC\left\langle w^r \right\rangle \cong \mC$, and $C^j_{d,\mC\left\langle w \right\rangle}$ is nothing but the complex
\beqn
0 \longra H^j\left(X,\caO_X\right) \stackrel{\wedge w}{\longra} H^j\left(X,\Omega^1_X\right) \stackrel{\wedge w}{\longra} \ldots \stackrel{\wedge w}{\longra} H^j\left(X,\omega_X\right),
\enqn
which is (complex-conjugate to) the {\em derivative complex} studied by Green and Lazarsfeld in \cite{GL1}.

Our main aim is to study the exactness of $C^j_{r,W}$. More precisely, we look for conditions on $W$ which guarantee that $C^j_{r,W}$ is exact in some (say $m$) of its first steps, (i.e., $C^j_{r,m,W}$ is exact), because this exactness will provide several inequalities between the Hodge numbers $h^{p,j}\left(X\right)$.

At some points, we will need to consider different subspaces $W$. Hence, we ``glue'' all the complexes (\ref{Complex-W}) with fixed $k = \dim W$, obtaining a complex of vector bundles on $\mG = \mG_k = Gr\left(k,V\right)$.

\begin{df}[Grassmannian BGG complex]
For any integers $r \geq 1$ and $0 \leq j \leq d$, the $(r,j)$-th {\em Grassmannian BGG complex} (of rank $k$) of $X$ is the complex of vector bundles on $\mG_k$
\begin{multline*} \label{BGG-complexes}
C^j_r: 0 \longra \Sym^rS\otimes H^j\left(X,\caO_X\right) \longra \Sym^{r-1}S \otimes H^j\left(X,\Omega_X^1\right) \longra \cdots \\
\cdots \longra \Sym^{r-i}S \otimes H^j\left(X,\Omega_X^i\right) \longra \cdots \\
\cdots \longra \Sym^{r-n}S \otimes H^j\left(X,\Omega_X^n\right)
\end{multline*}
where $n = \min\{r,d\}$ and over each point $W \in \mG_k$ it is given by (\ref{Complex-W}). Let $\caF^j_{r,n}$ denote the cokernel of the last map in $C^j_{r,n}$, the $(n,r,j)$-th {\em Grassmannian BGG sheaf} (of rank $k$) of $X$.
\end{df}

\begin{rmk}
If $k=1$, then $\mG = \mP = \mP\left(H^0\left(X,\Omega_X^1\right)\right)$, $S = \caO_{\mP}\left(-1\right)$ and $\Sym^rS= \caO_{\mP}\left(-r\right)$. So taking $k=1$ and $r=d$, the above complex is precisely (the complex-conjugate of) the BGG complex introduced by Lazarsfeld and Popa in \cite{LP}. In this way, the Grassmannian BGG complexes can be seen as generalizations of the BGG complex, with the new feature that they capture also the additive structure of the cohomology algebra of $X$. The sheaves $\caF_{r,n}^j$ generalize the BGG sheaves introduced in \cite{LP}.
\end{rmk}

The interest of studying these complexes is that, whenever they are exact at some point $W \in \mG$, they provide some inequalities involving some of the Hodge numbers $h^{i,j}\left(X\right)$. These inequalities are much stronger when the complex is exact {\em at every point}, so that the Grassmannian BGG sheaves are vector bundles and a deeper study of them is feasible (as we will do in Section \ref{sect:non-vanishing}). For example, the proof of the higher-dimensional Castelnuovo-de Franchis inequality given by Lazarsfeld and Popa in \cite{LP} is based on the fact that the BGG sheaf is an indecomposable vector bundle on $\mP^{q-1}$.

In order to study the exactness of $C^j_{r,W}$ we follow the ideas in Section 3 (A Nakano-type generic vanishing theorem) of \cite{GL1}. Consider the following complex of sheaves on $X$
\beqn
\caC_{r,W}: 0 \ra \Sym^rW \otimes \caO_X \ra \Sym^{r-1}W \otimes \Omega_X^1 \ra \cdots \ra \Sym^{r-i}W \otimes \Omega_X^i \ra \cdots \ra \Sym^{r-n}W \otimes \Omega_X^n
\enqn
where the maps $\mu_i: \Sym^{r-i}W \otimes \Omega_X^i \ra \Sym^{r-i-1}W \otimes \Omega_X^{i+1}$ are defined as in Definition \ref{df-complexes}. Clearly, its global sections form the complex $C^0_{r,W}$, and in general, $C^j_{r,W} = H^j\left(X,\caC_{r,W}\right)$ is the complex obtained by applying the $j$-th sheaf cohomology functor. Denote by $K^i = \Sym^{r-i}W \otimes \Omega_X^i$ the $i$-th term of $\caC_{r,W}$, and by $\caH^i=\caH^i\left(\caC_{r,W}\right)$ its $i$-th cohomology sheaf. Then there are two spectral sequences abutting to the hypercohomology of $\caC_{r,W}$, starting at
\beq \label{spectral_seq}
{'E}_1^{i,j}=H^j\left(X,K^i\right) = \Sym^{r-i}W \otimes H^j\left(X,\Omega_X^i\right) \qquad \mbox{and} \qquad {''E}_2^{i,j}=H^i\left(X,\caH^j\right).
\enq

The combined study of these spectral sequences leads to the wanted exactness of $C^j_{r,W}$ at some steps. We start with a generalization of Proposition 3.7 in \cite{GL1}, whose proof is analogous but more involved.

\begin{pro} \label{prop-degen}
For any $W \in \mG_k$, the spectral sequence ${'E}$ degenerates at ${'E}_2$, i.e. ${'E}_2 =\,{'E}_{\infty}$.
\end{pro}
\begin{proof}
We will denote by $A^{i,j}\left(X\right)$ the vector space of $\caC^{\infty}$ differential forms of type $\left(i,j\right)$, and will identify each cohomology class $\left[b\right] \in H^j\left(X,\Omega_X^i\right)$ with its only harmonic representative $b \in A^{i,j}\left(X\right)$. We will also use the $\partial\bar{\partial}$-Lemma (\cite{Voi1} Proposition 6.17): if $b \in A^{i,j}\left(X\right)$ is both $\partial$- and $\bar{\partial}$-closed, and either $\partial$- or $\bar{\partial}$-exact, then $b = \partial\bar{\partial}c = -\bar{\partial}\partial c$ for some $c \in A^{i-1,j-1}\left(X\right)$.

Fix $\{w_1,\ldots,w_k\}$ a basis of $W$, so that any element $b \in \Sym^{r-i}W \otimes H^j\left(X,\Omega_X^i\right)$ may be uniquely written as
$$b = \sum_{|J|=r-i}w_J \otimes \left[b_J\right],$$
where $J = \{1 \leq j_1 \leq j_2 \leq \cdots \leq j_{r-i} \leq k\}$, $w_J = w_{j_1} \cdots w_{j_{r-i}} \in \Sym^{r-i}W$ and $b_J \in A^{i,j}(X)$ is harmonic.

Firstly, we will show that the differential $d_2$ of ${'E}_2$ vanishes on every ${'E}_2^{i,j}$. By definition, any class in ${'E}_2^{i,j}$ is represented by some
$$b = \sum_{\left|J\right|=r-i}w_J \otimes \left[b_J\right] \in \ker\left\{\mu_i^j=H^j\left(\mu_i\right):\Sym^{r-i}W \otimes H^j\left(X,\Omega_X^i\right)\ra\Sym^{r-i-1}W \otimes H^j\left(X,\Omega_X^{i+1}\right)\right\},$$
that is, such that
$$\sum_{|J|=r-i}\sum_{s=1}^{r-i}w_{J-\left\{j_s\right\}}\otimes\left[w_{j_s} \wedge b_J\right] = \sum_{\left|J'\right|=r-i-1}w_{J'} \otimes \left[\sum_{j=1}^k w_j \wedge b_{J' \cup \left\{j\right\}}\right] = 0$$
where $J-\left\{j_s\right\}$ and $J' \cup \left\{j\right\}$ should be understood as operations on {\em multisets}. This last sum is zero if and only if all the classes $\left[\sum_{j=1}^k w_j \wedge b_{J' \cup \left\{j\right\}}\right]$ vanish in $H^j\left(X,\Omega_X^{i+1}\right) \cong H_{\bar{\partial}}^{i+1,j}\left(X\right)$ (considered as Dolbeault's cohomology classes), so we can assume that all the $\sum_{j=1}^k w_j \wedge b_{J' \cup \left\{j\right\}}$ are $\bar{\partial}$-exact. Since they are also both $\partial$- and $\bar{\partial}$-closed (because so are the $w_j$ and the $b_J$), there exist $c_{1,J'} \in A^{i,j-1}\left(X\right)$ such that
\beq \label{eq_1}
\sum_{j=1}^k w_j \wedge b_{J' \cup \left\{j\right\}} = \bar{\partial}\partial c_{1,J'},
\enq
and $d_2\left(b\right)$ is represented by
\begin{multline*}
\mu_{i+1}^{j-1}\left(\sum_{\left|J'\right|=r-i-1}w_{J'} \otimes \partial c_{1,J'}\right) = \sum_{\left|J'\right|=r-i-1}\sum_{s=1}^{r-i-1} w_{J'-\left\{j'_s\right\}} \otimes \left(w_{j'_s} \wedge \partial c_{1,J'}\right) = \\
= \sum_{\left|J''\right|=r-i-2} w_{J''} \otimes \left(\sum_{j=1}^k w_j \wedge \partial c_{1,J''\cup\left\{j\right\}}\right).
\end{multline*}
So we need to check that all the $a_{J''} = \sum_{j=1}^k w_j \wedge \partial c_{1,J''\cup\left\{j\right\}}$ are $\bar{\partial}$-exact (thus representing the zero class in $H_{\bar{\partial}}^{i+2,j-1}\left(X\right) \cong H^{j-1}\left(X,\Omega_X^{i+2}\right)$). On the one hand, note that $a_{J''} = -\partial\left(\sum_{j=1}^k w_j \wedge c_{1,J''\cup\left\{j\right\}}\right)$, so they are $\partial$-exact, and hence $\partial$-closed. On the other hand, using equation (\ref{eq_1}) we obtain
$$\bar{\partial}a_{J''} = -\sum_{j=1}^k w_j \wedge \bar{\partial}\partial c_{1,J''\cup\left\{j\right\}} = -\sum_{1 \leq j < l \leq k} \left(w_j \wedge w_l + w_l \wedge w_j\right) \wedge b_{J''\cup\left\{j,l\right\}} = 0,$$
so $a_{J''} = \bar{\partial}\partial c_{2,J''}$ for some $c_{2,J''} \in A^{i+1,j-2}\left(X\right)$. In particular, it is $\bar{\partial}$-exact and hence $d_2\left(b\right)=0$, as wanted.

To finish, we have to show that all the subsequent differentials $d_m$ also vanish. Assume inductively that for $2 \leq l < m$ we have $d_l=0$, and that for any $b$ as above we can find forms $c_{l,J_l} \in A^{i+l-1,j-l}\left(X\right)$ such that $\bar{\partial}\partial c_{l,J_l} = \sum_{j=1}^k w_j \wedge \partial c_{l-1,J_l\cup\left\{j\right\}}$ for every multisubset $J_l$ of $\left\{1,\ldots,r\right\}$ of cardinality $r-i-l$. Then, as before, $d_m\left(b\right)$ is the class in ${'E}_m^{i+m,j-m+1}={'E}_2^{i+m,j-m+1}$ of
$$\sum_{\left|J_m\right|=r-i-m} w_{J_m} \otimes \left(\sum_{j=1}^k w_j \wedge \partial c_{m-1,J_m\cup\left\{j\right\}}\right).$$
As above, the forms $\sum_{j=1}^k w_j \wedge \partial c_{m-1,J_m\cup\left\{j\right\}}$ are $\partial$-exact and $\bar{\partial}$-closed, so there exist forms $c_{m,J_m}$ as in the induction hypothesis, and in particular $d_r\left(b\right)=0$ because they are $\bar{\partial}$-exact.
\end{proof}

Suppose now that there is some integer $N$ such that $\caH^j=0$ for all $j<N$, or more generally $H^i\left(X,\caH^j\right)=0$ for $i+j < N$. Then we obtain ${''E}_2^{i,j}=0$ for all $i+j<N$, and hence, by (\ref{spectral_seq}) we get $\mH^m\left(X,\caC_{r,W}\right)=0$ for $m<N$. Looking at the other spectral sequence, it must hold ${'E}_{\infty}^{i,j}={'E}_2^{i,j}=0$ for all $i+j<N$. But ${'E}_2^{i,j}$ is precisely the cohomology of $C^j_{r,W}$ at the $i$-th step, so we get that $C^j_{r,W}$ is exact in the first $N-j$ steps. In particular, $C^0_{r,W}$ is exact at $W \in \mG$ in the first $N$ steps.

Therefore, we will next try to answer the next

\begin{que} \label{que-Eagon}
Fixed $N$, under which hypothesis on $W$ can we assure $H^i(X,\caH^j)=0$ for $i+j<N$?
\end{que}

For this purpose, we will first try to identify the sheaves $\caH^j$. Consider the dual to the evaluation map,
$$g: T_X = \left(\Omega^1_X\right)^{\vee} \longra W^{\vee} \otimes \caO_X,$$
and denote by $\caK = \coker\left(g\right)$. For any $i=1,\ldots,k=\dim W$, let
$$Z_i = Z_i\left(W\right) = \left\{p \in X \,\vert\, \rk \left(g_p:T_{X,p} \ra W^{\vee} \right)< i\right\} = \left\{p \in X \,\vert\, \rk\left(ev_p: W \ra T_{X,p}^{\vee}\right)<i\right\}$$
be the locus where the forms in $W$ span a subspace of dimension $<i$ of the cotangent space, or where the kernel of the evaluation map has dimension greater than $k-i$. Clearly, $\caK$ is supported on $Z_k$, the locus where $g$ is not surjective.

\begin{df}[Non-degenerate subspace] \label{df-non-deg}
We say that $W \subseteq H^0\left(X,\Omega_X^1\right)$ is {\em non-degenerate} if
$$\codim Z_i \geq d-i+1 \qquad \forall \, 1 \leq i \leq \min\left\{k,d\right\}.$$
\end{df}

\begin{rmk}
In the case $k \leq d$, there is a slightly weaker condition which is enough for our purposes, but has the inconvenient that depends on the $r$ we are considering. We will not use it in the sequel, but we include it for the sake of completeness: $W$ is non-degenerate {\em on degree $r$} if
\begin{itemize}
\item $\codim Z_k = d-k+1$ in the case $r \leq d-k+1$, or
\item $\codim Z_j \geq d-j+1$ for $j=\max\{1,d-r+1\},\ldots,k$ in the case $r \geq d-k+2$.
\end{itemize}
\end{rmk}

Definition \ref{df-non-deg} is motivated by some results on complexes of Eagon-Northcott type which allow to identify the cohomology sheaves $\caH^i$ of $\caC_{r,W}$ for non-degenerate $W$ (see \cite{Av}, \cite{BV} and \cite{Eis}, Appendix A2.6 for more details on these kind of complexes).

\begin{lem} \label{lem_ext}
Fix any $r \geq 1$, and assume that $W$ is non-degenerate. Then $\caH^i(\caC_{r,W})=\caExt^i_{\caO_X}(\Sym^r \caK,\caO_X)$ for all $0 \leq i < r$.
\end{lem}
\begin{proof}
The last $r$ steps of the $r$-th Eagon-Northcott complex $EN_r$ associated to $g$ look like
$$EN_r: \ldots \longra \left(\Omega_X^r\right)^{\vee} \longra \left(\Omega_X^{r-1}\right)^{\vee} \otimes W^{\vee} \longra \cdots \longra \left(\Omega_X^1\right)^{\vee} \otimes \Sym^{r-1}W^{\vee} \longra \caO_X \otimes \Sym^r W^{\vee}.$$
The non-degeneracy of $W$ implies (\cite{Laz1}, Theorem B.2.2 for the case $k \leq d$, and \cite{Av}, Proposition 3.(3) for the case $k \geq d$) that $EN_r$ is the end of a locally free resolution of $\Sym^r \caK$, so we can compute
$$\caExt^i_{\caO_X}\left(\Sym^r \caK, \caO_X\right) = \caH^i\left(\caHom_{\caO_X}\left(EN_r,\caO_X\right)\right).$$
But clearly the first $r$ steps of $\caHom_{\caO_X}\left(EN_r,\caO_X\right)$ form the complex $\caC_{r,W}$ (with maps divided by some factorials), and the claim follows.
\end{proof}

We now focus on the case $k \leq d$, where some well-known properties of the $\caExt$ sheaves lead to a first result:

\begin{thm} \label{thm-exact-k<d}
If $W$ is non-degenerate, then the complex
\beqn 
C_{r,W}^j: 0 \ra \Sym^r W \otimes H^j\left(X,\caO_X\right) \ra \cdots \ra \Sym^{r-i} W \otimes H^j\left(X,\Omega_X^i\right) \ra \cdots \ra \Sym^{r-n} W \otimes H^j\left(X,\Omega_X^n\right)
\enqn
is exact at least in the first $d-k-j+1$ steps.
\end{thm}
\begin{proof}
For any coherent sheaf $\caF$ on $X$ we have (see \cite{HL}, Proposition 1.6.6)
$$\caExt^i(\caF,\caO_X) = 0 \qquad \forall \, i < \codim \Supp \caF.$$
Since $\Supp \Sym^r \caK = \Supp \caK = Z_k$ has codimension at least $d-k+1$ because $W$ is non-degenerate, we obtain
$$\caH^j(\caC_{r,W}) = \caExt^j(\Sym^r \caK,\caO_X) = 0$$
for all $j \leq d-k$. Therefore the second spectral sequence in (\ref{spectral_seq}) satisfies ${''E}_2^{i,j}=0$ for all $i$ and all $j \leq d-k$. Since ${''E}_2^{i,j}$ abuts to the hypercohomology of $\caC_{r,W}$, this implies that $\mH^i(X,\caC_{r,W}) = 0$ for all $i \leq d-k$. Recalling that the first spectral sequence ${'E}_1^{i,j}$ degenerates at ${'E}_2$ (Proposition \ref{prop-degen}), and it also abuts to the hypercohomology of $\caC_{r,W}$, this implies that ${'E}_2^{i,j}=0$ for all $i+j \leq d-k$. But ${'E}_2^{i,j}$ is precisely the cohomology of the complex $C^j_{r,W}$ at the $i$-th step, so the claim follows.
\end{proof}

Some known results suggest that $C_{r,W}^j$ should be exact under weaker hypothesis, and even for some $k>d$. To get such a result we should study the cohomology of the sheaves $\caH^i=\caExt_{\caO_X}^i\left(\Sym^r\caK,\caO_X\right)$, which may vanish even if the sheaves do not. For instance, in general, the approach with spectral sequences shows that the kernel of the first map of $C_{r,W}^0$, $\mu_0^0: \Sym^r W \ra \Sym^{r-1} W \otimes H^0\left(X,\Omega_X^1\right)$ is $H^0\left(X,\caHom\left(\Sym^r\caK,\caO_X\right)\right)$, which must always vanish because $\mu_0^0$ is always injective. Furthermore, according to \cite{VGA1}, ${'E}_2^{1,0}$ vanishes for general $W$ of even dimension $k$ if $X$ is not fibred over an Albanese general type variety of dimension at most $\frac{k}{2}$ (more generally, if $X$ has no generalized Lagrangian form of rank $\frac{k}{2}$).

Moreover, as the following example shows, the spectral sequence ${''E}_2$ is not degenerate in general. Therefore, even if the cohomologies ${''E}_2^{i,j}$ of $\caH^i$ do not vanish, the limit groups ${''E}_{\infty}^{i,j}$ may anyway vanish, so Theorem \ref{thm-exact-k<d} is not sharp.

\begin{ex}
Consider $C_1,C_2 \subset \mP^2$ two smooth curves of degree 4 (genus 3) intersecting transversely in 16 points $p_1,\ldots,p_{16}$, and let $X = C_1 \times C_2$. Fix a basis $\eta_1,\eta_2,\eta_3 \in H^0\left(\mP^2,\caO_{\mP^2}\left(1\right)\right)$, and denote by $\alpha_i$ and $\beta_i$ its restrictions to $C_1$ and $C_2$ respectively, which can be thought as differential forms since $\omega_{C_i} \cong \caO_{C_i}(1)$ by adjunction. Finally, set $w_i = p_1^*\alpha_i + p_2^*\beta_i \in H^0\left(X,\Omega_X^1\right)$, let $W \subset H^0\left(X,\Omega_X^1\right)$ be the vector space spanned by the $w_i$, and consider the case $r=2$:
\beq \label{ex1}
\caC_{2,W}: \, 0 \longra \Sym^2 W \otimes \caO_X \longra W \otimes \Omega_X^1 \longra \omega_X \longra 0.
\enq

The situation is explicit enough to compute most of the objects considered above. An immediate computation shows that $Z_1 = \emptyset$ and $Z_2=\left\{P_1,\ldots,P_{16}\right\}$, where $P_i=\left(p_i,p_i\right)$, so $W$ is non-degenerate. Moreover, a complete description of the first spectral sequence ${'E}_1$ can be carried out to find that ${'E}_2^{i,j}=0$ for all $i,j$ except for ${'E}_2^{0,2} \cong \mC^{37}$, ${'E}_2^{1,1} \cong \mC^{18}$ and  ${'E}_2^{1,0} \cong \mC^3$. This implies that $\mH^1\left(X,\caC_{2,W}\right)\cong\mC^3$, $\mH^2\left(X,\caC_{2,W}\right)\cong\mC^{55}$, and all the other hypercohomology groups vanish.

As for the second spectral sequence, we start computing the cohomology sheaves $\caH^i$ of (\ref{ex1}). The last map is surjective, hence $\caH^2 = 0$. The sheaf $\caH^1$ is supported on $Z_2$, and the transversality of $C_1$ and $C_2$ implies that each stalk $\caH^1_{P_i}$ is a three-dimensional vector space, so that $H^0\left(X,\caH^1\right) \cong \mC^{48}$ and the rest of its cohomology groups are zero. This computation is enough to show that ${''E}_2$ is not degenerate, since if it was, the group $H^0\left(X,\caH^1\right)\cong\mC^{48}$ would be a summand of $\mH^1\left(X,\caC_{2,W}\right)\cong\mC^3$.
\end{ex}

\begin{rmk}
In the case $\dim W=1$, the exactness of $C^0_{d,W}$ is directly related to the {\em cohomological support loci}
$$V^i\left(X,\omega_X\right) = \left\{\alpha \in \Pic^0\left(X\right)\,|\,h^i\left(X,\omega_X\otimes\alpha\right)\neq0\right\} \subseteq \Pic^0\left(X\right)$$
introduced by Green-Lazarsfeld in \cite{GL1}. More precisely, if $W=\mC\left\langle w\right\rangle \subseteq H^0\left(X,\Omega_X^1\right)$, then $C^0_{d,W}$ is complex-conjugate to the {\em derivative complex}
\beqn
0 \longra H^0\left(X,\caO_X\right) \stackrel{\wedge \overline{w}}{\longra} H^1\left(X,\caO_X\right) \stackrel{\wedge \overline{w}}{\longra} \cdots \stackrel{\wedge \overline{w}}{\longra} H^d\left(X,\caO_X\right),
\enqn
which is exact at $H^i\left(X,\caO_X\right)$ if and only if the ``line'' in $\Pic^0\left(X\right)$ spanned by $\overline{w}\in H^1\left(X,\caO_X\right) \cong T_{\caO_X}\Pic^0\left(X\right)$ is not contained in $V^i\left(X,\omega_X\right)$ (\cite{GL2} Corollary 3.3). Equivalently, since the cohomological support loci are translates of subtori of $\Pic^0\left(X\right)$ (\cite{GL2},Theorem 0.1), we can say that $C^0_{d,W}$ is exact at the $i$-th step if and only if $\overline{W}$ is not tangent to $V^i\left(X,\omega_X\right)$ at $\caO_X$.

This is no longer true for higher-rank derivative complexes. Indeed, in the previous example, the cohomological support loci of $X$ are $V^1\left(X,\omega_X\right) = \pi_1^*\Pic^0\left(C_1\right) \cup \pi_2^*\Pic^0\left(C_2\right)$ and $V^2\left(X,\omega_X\right) = \left\{\caO_X\right\}$, which are clearly transverse to $\overline{W}$ at $\caO_X \in \Pic^0\left(X\right)$ while $C^0_{2,W}$ is exact only at the $0$-th step.
\end{rmk}

We now turn to the numerical consequences of Theorem \ref{thm-exact-k<d}.

\begin{cor} \label{cor-k<d-1}
If $X$ admits a non-degenerate subspace of dimension $k \left(\leq d\right)$, then
\beq \label{eq-cor-exact-1}
\sum_{i=0}^p (-1)^{p-i}\binom{r-i+k-1}{k-1}h^{i,j}(X) \geq 0
\enq
for every $p \leq \min\left\{d-k-j+1,r\right\}$. In particular
\beqn 
h^{p,j}(X) \geq \sum_{i=0}^{p-1} (-1)^{p-i-1}\binom{p-i+k-1}{k-1}h^{i,j}(X)
\enqn
for every $p+j \leq d-k+1$.
\end{cor}
\begin{proof}
The first inequality is a direct consequence of Theorem \ref{thm-exact-k<d}, and the second one is the particularization to the case $r=p$.
\end{proof}

And computing a little bit more we find the next (more explicit) result:

\begin{cor} \label{cor-k<d-2}
If $X$ admits a non-degenerate subspace of dimension $k \leq d$, then
$$h^{p,j}(X) \geq \binom{k}{p}h^{0,j}(X)$$
for every $p \leq k$ and $p \leq d-k-j+1$, and therefore
$$h^{p,j}(X) \geq \binom{k}{p}\binom{k}{j}$$
if $p,j \leq k$ and $p+j \leq d-k+1$.
\end{cor}
\begin{proof}
It is a consequence of the identity
\beq \label{combin}
\sum_{n=0}^{\min\{A,B\}} \left(-1\right)^{B-i}\binom{A}{n} \binom{A+B-n-1}{B-n} =
\begin{cases}
1 & \text{ if } B=0 \\
0 & \text{ otherwise}
\end{cases}
\enq
which holds for any non-negative integers $A,B$ and can be easily proved by looking at the coefficient of $x^B$ in the expansion of the right-hand side of
\beqn
1 = \frac{\left(1+x\right)^A}{\left(1+x\right)^A} = \left(\sum_{n=0}^A\binom{A}{n}x^n\right)\left(\sum_{m\geq0}\left(-1\right)^m\binom{A+m-1}{m}x^m\right).
\enqn

Indeed, denote by $M_{p,j}=\sum_{i=0}^p \left(-1\right)^{p-i} \binom{p-i+k-1}{k-1}h^{i,j}(X)$, the right-hand-side of (\ref{eq-cor-exact-1}), and compute
\begin{multline*}
\sum_{i=0}^p \binom{k}{p-i} M_{i,j} = \sum_{i=0}^p \binom{k}{p-i} \sum_{m=0}^i \left(-1\right)^{i-m}\binom{i-m+k-1}{k-1} h^{m,j}(X) = \\
= \sum_{m=0}^p \left(\sum_{i=m}^p\left(-1\right)^{i-m}\binom{k}{p-i}\binom{i-m+k-1}{i-m}\right) h^{m,j}(X) = h^{p,j}(X),
\end{multline*}
where the last equality follows from (\ref{combin}) because
$$\sum_{i=m}^p\left(-1\right)^{i-m}\binom{k}{p-i}\binom{i-m+k-1}{i-m} = \sum_{n=0}^{p-m}\left(-1\right)^{p-m-n}\binom{k}{n}\binom{p-n-m+k-1}{p-n-m}$$
and $p-m \leq p \leq k$. Therefore,
$$0 \leq \sum_{i=1}^p \binom{k}{p-i} M_{i,j} = h^{p,j}(X) - \binom{k}{p}M_{0,j} = h^{p,j}(X) - \binom{k}{p}h^{0,j}(X),$$
as wanted. The second statemet follows at once from the first statement applied to $h^{0,j}(X)=h^{j,0}(X)$.
\end{proof}


\section{Subvarieties of Abelian varieties}
\label{sect:subvAV}

In this section we focus on subvarieties of Abelian varieties, for which generic subspaces $W \subseteq H^0\left(X,\Omega_X^1\right)$ are non-degenerate (see Proposition \ref{pro-non-deg-subv}) and it is possible to apply the results in the previous section.

\begin{pro} \label{pro-non-deg-subv}
Let $X \subseteq A$ be a smooth subvariety of an Abelian variety $A$. Then, for every $k=1,\ldots,q(X)$, the $k$-dimensional non-degenerate subspaces $W \in Gr\left(k,H^0\left(X,\Omega_X^1\right)\right)$ form a non-empty Zariski-open subset.
\end{pro}
\begin{proof}
First of all, upon replacing $A$  by the subtorus spanned by $X$, we may assume that $V=H^0\left(X,\Omega_X^1\right)\cong H^0\left(A,\Omega_A^1\right)$. Since the non-degeneracy condition is open, we only need to construct a non-degenerate subspace of any dimension $k$. We will proceed by induction over $k$.

A one-dimensional subspace $W = \mC\left\langle w \right\rangle$ is non-degenerate if and only if $\codim Z_1 \geq d$. Since $Z_1 = Z(w)$ is the set of zeroes of any generator $w$, $W$ is non-degenerate if and only if $w$ vanishes (at most) at isolated points. To prove that generic elements $w \in V$ satisfy that, let us consider the incidence variety
$$I=\left\{\left(x,\left[w\right]\right)\in X \times \mP\left(V\right)\,\vert\, w(x)=0\right\} \subseteq X \times \mP\left(V\right).$$
The first projection makes $I$ into a projective bundle of fibre $\mP^{q-d-1}$ (where as usual, $d=\dim X$ and $q = q(X)$). Indeed, the fibre over any $x \in X$ is (the projectivization of) the set of 1-forms vanishing at $x$. Since the tangent space $T_{X,x}$ injects into $T_{A,x}$, the set of 1-forms vanishing at $x$ is the annihilator $T_{X,x}^{\perp}$ inside $T_{A,x}^{\vee} \cong V$, which has dimension $q-d$. In particular, $I$ is irreducible of dimension $(q-d-1)+d=q-1$.

Consider now the second projection $I \ra \mP\left(V\right)$. It is clear that the fibre over a point $\left[w\right]$ is the zero set $Z\left(w\right)$, so we want to see that a general fibre has dimension at most 0. If $I$ dominates $\mP\left(V\right) \cong \mP^{q-1}$, the general fibre has dimension $(q-1)-(q-1)=0$. If otherwise $I$ does not dominate $\mP\left(V\right)$, the general fibre is empty (that is, a generic 1-form does not vanish at any point). In any case, we are done.

For the inductive step, note first that if we have two nested subspaces $W' \subseteq W \subseteq V$ and $\dim W' = k'$, then $Z_i\left(W\right) \subseteq Z_i\left(W'\right)$ for every $i=1,\ldots,k'$. Therefore, if $W'$ is non-degenerate and $k=\dim W = k'+1$, then $\codim Z_i\left(W\right) \geq \codim Z_i\left(W'\right) \geq d-i+1$ for $i=1,\ldots,k-1$, and $W$ will be non-degenerate as soon as $\codim Z_k\left(W\right) \geq d-k+1$.

Fix a non-degenerate subspace $W'$ of dimension $k-1$ (it exists by the induction hypothesis), so that in particular $\codim Z_{k-1}(W') \geq d-k+2$, and let $X' = X - Z_{k-1}\left(W'\right)$ be the open set where the evaluation $W' \ra T_{X,x}^{\vee}$ is injective. For any $x \in X'$, denote by $W'_x \subseteq T_{X,x}^{\vee}$ the image of the evaluation, and by $E_x \subseteq T_{X,x}$ the subspace of tangent vectors annihilated by $W'_x$, which has dimension $\dim T_{X,x}-\dim W'_x = d-k+1$. Consider the new incidence variety
$$I_k = \left\{\left(x,W\right)\,\vert\,x\in X', W=W'+\mC\langle w\rangle, E_x \subseteq \ker w(x)\right\} \subseteq X' \times \mP\left(V/W'\right).$$
Note that the condition $E_x\subseteq \ker w(x)$ is independent of the choice of the complement $\mC\langle w \rangle$ of $W'$ in $W$, so $I_k$ is well defined. As in the case $k=1$, the first projection makes $I_k$ into a $\mP^{q-d-1}$-bundle, so $I_k$ is irreducible of dimension $q-1$. Indeed, the fibre over a point $x \in X'$ is the projectivization of
$$\left\{w+W' \in V/W'\,\vert\,E_x \subseteq \ker w\right\}=\left\{w \in V \,\vert\,E_x \subseteq \ker w\right\}/W'=E_x^{\perp}/W' \cong \mC^{q-d},$$
where the annihilator $E_x^{\perp}$ is taken in $V$, that is, it is the kernel of the restriction $V \twoheadrightarrow E_x^{\vee}$ dual to the composition of inclusions $E_x \subseteq T_{X,x} \subseteq T_{A,x}=V^{\vee}$.

As for the second projection, the fibre over $W=\mC\langle w \rangle + W' \in \mP\left(V/W'\right)$ is the set
$$\left\{x \in X' \,\vert\, E_x \subseteq \ker w(x)\right\}=\left\{x \in X' \,\vert\, w(x)\in W'_x\right\}=Z_k(W)\cap X'=Z_k(W)-Z_{k-1}(W'),$$
and for $W$ generic its dimension is either zero (if the second projection is not dominant) or $\dim I_k - \dim \mP\left(V/W'\right) = (q-1)-(q-(k-1)-1)=k-1$. Since the dimension of $Z_{k-1}(W')$ is at most $k-2$, we conclude that $\dim Z_k(W) \leq k-1$ for $W$ generic containing $W'$, finishing the proof.
\end{proof}

\begin{rmk}
Note that the only property we have used is that the tangent spaces $T_{X,x}$ inject into the tangent space of the Abelian variety at every point. Therefore, the same result holds true for \'etale coverings of subvarieties of Abelian varieties.
\end{rmk}

Therefore we can apply corollaries \ref{cor-k<d-1} and \ref{cor-k<d-2} for any $k \leq d$ to obtain in particular the next inequality:

\begin{cor}
If $X$ is a subvariety of an Abelian variety and $p,j \geq 0$ satisfy $\max\{p,j\} \leq d+1-(p+j)$, then
$$h^{p,j}(X) \geq \binom{d+1-(p+j)}{p}\binom{d+1-(p+j)}{j}.$$
\end{cor}

For $X$ a subvariety of an Abelian variety $A$ with $H^0\left(X,\Omega_X^1\right) = H^0\left(A,\Omega_A^1\right)$ it is also useful to consider the extremal case $k = q$, that is, $W = H^0\left(X,\Omega_X^1\right)$ is the whole space of holomorphic 1-forms. In this case, the cokernel $\caK$ of the previous section is simply the normal bundle $N_{X/A}$. Since it is a vector bundle, so is $\Sym^r\caK$, and hence $\caExt^i_{\caO_X}\left(\Sym^r\caK,\caO_X\right) = 0$ for every $i>0$. Therefore, the second spectral sequence ${''E}$ is degenerate at ${''E}_2$, and its only possibly non-zero terms are ${''E}_2^{i,0}=H^i\left(X,\Sym^r N_{X/A}^{\vee}\right)$ (recall the definition (\ref{spectral_seq}) of ${''E}$ and Lemma \ref{lem_ext}). This leads to the following

\begin{pro} \label{prop-q}
Let $X \subseteq A$ be a subvariety of an Abelian variety such that $H^0\left(X,\Omega_X^1\right) = H^0\left(A,\Omega_A^1\right)$. If for some positive integers $r,N$ the normal bundle $N_{X/A}$ satisfies $H^i\left(X,\Sym^r N_{X/A}^{\vee}\right)=0$ for all $i < N$, then the complex
\begin{equation} \label{complex-q}
0 \ra \left(\Sym^r H^0\left(X,\Omega_X^1\right)\right) \otimes H^j\left(X,\caO_X\right) \ra 
\cdots \ra \left(\Sym^{r-N+j} H^0\left(X,\Omega_X^1\right)\right) \otimes H^j\left(X,\Omega_X^{N-j}\right)
\end{equation}
is exact for any $j<N$.
\end{pro}
\begin{proof}
By the previous discussion, since $\Sym^r N_{X/A}$ is locally free, the spectral sequence
$${''E}_2^{i,j}=H^i\left(X,\caExt_{\caO_X}^j\left(\Sym^r N_{X/A},\caO_X\right)\right) \Ra \mH^n\left(X,\caC_{r,V}\right)$$
is degenerate and gives $\mH^i\left(X,\caC_{r,V}\right) = H^i\left(X,\Sym^r N_{X/A}^{\vee}\right)=0$ for any $i<N$.
These vanishings, combined with Proposition \ref{prop-degen}, imply the vanishing of ${'E}_2^{i,j}$ for all $i+j<N$. Recalling that ${'E}_2^{i,j}$ is the cohomology of $C_{r,V}^j$ at the $i$-th step, the claim follows directly.
\end{proof}

\begin{cor} \label{cor-q}
If $X$ is as in the above Proposition, then
\beqn 
\sum_{i=0}^p (-1)^{p-i}\binom{r-i+q\left(X\right)-1}{q\left(X\right)-1}h^{i,j}(X) \geq 0
\enqn
for all $p + j\leq N$. If furthermore $H^i\left(X,\Sym^r N_{X/A}^{\vee}\right)=0$ for all $0 < r < N-j$, then 
$$h^{i,j}\left(X\right) \geq \binom{q\left(X\right)}{i}\binom{q\left(X\right)}{j}$$
for all $i+j\leq n$.
\end{cor}
\begin{proof}
The first assertion follows at once by dimension counting on (\ref{complex-q}), and the second one follows as in the proof of Corollary \ref{cor-k<d-2}.
\end{proof}

The main drawback of Proposition \ref{prop-q} is the difficulty to check the vanishing of $H^i\left(X,\Sym^r N_{X/A}^{\vee}\right)$. However, as the next example shows, some smooth intersections of ample divisors satisfy the vanishing of $H^i\left(X,\Sym^r N_{X/A}^{\vee}\right)$ for all $i < \dim X$, and the inequalities of Corollary \ref{cor-q} are sharp for $i+j<n$.

\begin{ex}
Let $D_1,\ldots,D_c \subseteq A$ be ample divisors on an Abelian variety such that the partial intersections $X_k = D_1 \cap \ldots \cap D_k$ are smooth, and let $X = X_c$. Then $H^i\left(X,\Sym^rN_{X/A}^{\vee}\right) = 0$ for every $r > 0$ and $i < \dim X = q\left(A\right)-c$. Moreover,
\beqn
h^{i,j}\left(X\right) = \binom{q\left(X\right)}{i}\binom{q\left(X\right)}{j}
\enqn
as long as $i+j < \dim X$.

As for the vanishing of the cohomology groups, one shows by induction on $k$ and $r$ that more generally $H^i\left(X_k,\left(\Sym^rN_{X_k/A}^{\vee}\right)\left(-D\right)\right) = 0$ for all $r>0, i<\dim X_k$, where $D$ is either 0 or an ample divisor on $A$. Indeed, for $k=1$ we have $N_{X_1/A}^{\vee} \cong \caO_{X_1}\left(-X_1\right)$ and $\Sym^rN_{X_1/A}^{\vee} \cong \caO_{X_1}\left(-rX_1\right)$, and the long exact cohomology sequence of
$$0 \longra \caO_A\left(-\left(r+1\right)X_1-D\right) \longra \caO_A\left(-rX_1-D\right) \longra \Sym^rN_{X_1/A}^{\vee}\left(-D\right) \cong \caO_{X_1}\left(-rX_1-D\right) \longra 0$$
combined with the Kodaira vanishing theorem gives the assertion for $k=1, r>0$. For bigger $k$, one first proves the assertion for $r=1$ as follows: Kodaira vanishing applied to
\beq
\label{induction-step}
0 \longra N_{X_{k-1}/A}^{\vee}\left(-D\right)_{|X_k} \longra N_{X_k/A}^{\vee}\left(-D\right) \longra N_{X_k/X_{k-1}}^{\vee}\left(-D\right) \cong \caO_{X_k}\left(-D_k-D\right) \longra 0
\enq
implies that
\beq
\label{first-iso}
H^i\left(X_k,N_{X_{k-1}/A}^{\vee}\left(-D\right)_{|X_k}\right) \cong H^i\left(X_k,N_{X_k/A}^{\vee}\left(-D\right)\right)
\enq
for all $i < \dim X_k$. Now the induction hypothesis applied to the exact sequence
\beq
\label{restriction}
0 \longra N_{X_{k-1}/A}^{\vee}\left(-D-D_k\right) \longra N_{X_{k-1}/A}^{\vee}\left(-D\right) \longra N_{X_{k-1}/A}^{\vee}\left(-D\right)_{|X_k} \longra 0
\enq
gives the vanishing of the groups in (\ref{first-iso}), since
$$H^i\left(X_{k-1},N_{X_{k-1}/A}^{\vee}\left(-D-D_k\right)\right) = H^i\left(X_{k-1},N_{X_{k-1}/A}^{\vee}\left(-D\right)\right) = 0$$
for all $i < \dim X_{k-1} = \dim X_k +1$. It remains to prove the inductive step for $r$, which follows in the same way by using the exact sequence
$$0 \longra \left(\Sym^rN_{X_{k-1}/A}^{\vee}\right)\left(-D\right)_{|X_k} \longra \left(\Sym^rN_{X_k/A}^{\vee}\right)\left(-D\right) \longra \left(\Sym^{r-1}N_{X_k/A}^{\vee}\right)\left(-D-D_k\right) \longra 0$$
induced by (\ref{induction-step}), and the corresponding analogue to (\ref{restriction}).

\end{ex}


\section{More on $h^{2,0}(X)$}
\label{sect:h20}

In \cite{VGA1} we proved a lower bound for the $h^{2,0}$ of an irregular variety of any dimension without higher irrational pencils. In this section we will compare it with the inequalities obtained in Corollaries \ref{cor-k<d-1} and \ref{cor-k<d-2}.

To be precise, in \cite{VGA1} we proved that if $X$ does not admit any higher irrational pencil, then the complex
\beqn 
0 \longra \Sym^2 W \longra W \otimes H^0\left(X,\Omega_X^1\right) \longra H^0\left(X,\Omega_X^2\right)
\enqn
is exact for generic $W \subseteq H^0\left(X,\Omega_X^1\right)$ of even dimension $\dim W = 2k' < 2\dim X$. This exactness gives the inequalities
\beq \label{ineq-VGA1}
h^{2,0}(X) \geq 2k'q(X)-\binom{2k'+1}{2} \qquad \forall \, k'<d,
\enq
and taking the maximum over all possible $k'$ we obtained the final
\begin{thm}[\cite{VGA1}, Theorem 1.1] \label{thm-VGA1}
Let $X$ be an irregular variety without higher irrational pencils. Then it holds
\beqn 
h^{2,0}\left(X\right) \geq
\begin{cases}
\binom{q\left(X\right)}{2} & \text{ if } q\left(X\right) \leq 2 \dim X -1, \\
2\left(\dim X-1\right)q\left(X\right) - \binom{2\dim X-1}{2} & \text otherwise.
\end{cases}
\enqn
\end{thm}

In order to obtain such a result with the techniques of the present article, we must use Theorem \ref{thm-exact-k<d} for the case $r=2,j=0$. One problem that overcomes is the different nature of the hypothesis. Indeed, if the Albanese map of the variety is ramified, there is no obvious relation between the existence of non-degenerate subspaces and the non-existence of fibrations over varieties of Albanese general type.

As for the inequalities, if $q\left(X\right) \geq 2\dim X$, the strongest case of (\ref{ineq-VGA1}) is obtained for $k'=d-1$, hence $k=2d-2$. Such an inequality is impossible to obtain with Theorem \ref{thm-exact-k<d}, since it requires $k\leq d-1$, which is very far from $k=2d-2$. However, as we will see next, it is possible to obtain better bounds in at least two ways (with stronger hypotheses).

\subsection{Bounds from non-vanishing of Chern classes}
\label{sect:non-vanishing}

Assume first that Theorem \ref{thm-exact-k<d} holds for $r=2,j=0$ and {\em every} subspace $W \in \mG = Gr\left(k,H^0\left(X,\Omega_X^1\right)\right)$ for some fixed $k \leq d-1$. In this case, the Grassmannian BGG complex on $\mG$
\beq \label{Grass-c20}
C_2^0: \quad 0 \ra \Sym^2 S \ra S \otimes H^0\left(X,\Omega_X^1\right) \ra H^0\left(X,\Omega_X^2\right) \otimes \caO_{\mG} \ra \caF_{2,2}^0 \ra 0
\enq
is {\em everywhere} exact, so the cokernel $\caF = \caF_{2,2}^0$ is also a vector bundle. If we were able to compute the (total) Chern class of $\caF$, $c\left(\caF\right)$, we would obtain estimates on $\rk \caF$ which in turn will give lower bounds on 
\beqn
h^{2,0}\left(X\right) = \rk\left(\caF\right) + kq - \binom{k+1}{2}.
\enqn

Suppose for a moment that the Chern class of $\caF$ of degree $\dim \mG = k\left(q-k\right)$ is non-zero. This would imply that $\caF$ has rank at least $k\left(q-k\right)$, and therefore
\beq \label{second-bound}
h^{2,0}\left(X\right) \geq k\left(q-k\right) + kq - \binom{k+1}{2} = 2kq - \left(k^2 + \binom{k+1}{2}\right),
\enq
which has the same asymptotic behaviour as (\ref{ineq-VGA1}). Furthermore, since $k^2 + \binom{k+1}{2} < \binom{2k+1}{2}$, we would obtain a slightly stronger bound.

The problem is now reduced to compute the Chern class
\beqn 
c\left(\caF\right) = \frac{c\left(H^0\left(X,\Omega_X^2\right) \otimes \caO_{\mG}\right) c\left(\Sym^2 S\right)}{c\left(S \otimes H^0\left(X,\Omega_X^1\right)\right)} = c\left(\Sym^2 S\right) c\left(Q\right)^q,
\enqn
(since $c\left(S\right)^{-1} = c\left(Q\right)$). In general, this computation turns out to be very complicated. Indeed, although the power $c\left(Q\right)^q$ is easy to describe in terms of the Schubert classes of $\mG$, the formula for the Chern class of a symmetric power of some vector bundle $E$ depends on the rank of $E$, and we do not know of any explicit computation, even in the (rather concrete) case of tautological bundles over a Grassmannian.

Therefore, we have been forced to make explicit computations fixing both $k=2,3,4$ and $q = k+1,\ldots,12$. In these cases, the Chern classes of $\caF$ of highest degree vanish, hence the bounds (\ref{second-bound}) are out of reach with this last method. Furthermore, there is some pattern in the Schubert classes whose coefficient is non-zero, which leads us to formulate the following conjecture (recall the notation for Schubert classes introduced at the beginning).

\begin{conj} \label{conj-h20}
Let $\mu$ be the partition $\left(q-k-1,q-k-2,q-k-3,\ldots,q-2k\right)$ if $q \geq 2k$, and $\left(q-k-1,q-k-2,\ldots,1,0,\ldots\right)$ if $q<2k$. The coefficient of the Schubert class $\sigma_{\lambda}$ in $c(\caF)$ is zero for every $\lambda$ bigger than\footnote{We say that a partition $\left(\lambda_1\geq\ldots\geq\lambda_k\right)$ is bigger than $\left(\mu_1\geq\ldots\geq\mu_k\right)$ if $\lambda_i \geq \mu_i \, \forall i$, with strict inequality for some $i$.} $\mu$, while the coefficient of $\sigma_{\mu}$ is non-zero.
\end{conj}

As we said, we have checked Conjecture \ref{conj-h20} for $k=2,3,4$ and $q=k+1,\ldots,12$. In any case, if Conjecture \ref{conj-h20} holds true, then we obtain the following

\begin{pro} \label{pro-conj-h20}
If the Grassmannian BGG complex (\ref{Grass-c20}) of an irregular variety $X$ is everywhere exact and Conjecture \ref{conj-h20} holds, then
\beqn
h^{2,0}\left(X\right) \geq \begin{cases} \binom{q}{2} \quad \mbox{if } q \leq 2k, \\ 2kq - \binom{2k+1}{2} \quad \mbox{if } q \geq 2k. \end{cases}
\enqn
\end{pro}
\begin{proof}
Computing the codimension of $\sigma_{\mu}$ we obtain
\beqn
\rk\left(\caF\right) \geq \begin{cases} \binom{q-k}{2} \quad \mbox{if } q \leq 2k, \\ \frac{k\left(2q-3k-1\right)}{2} \quad \mbox{if } q \geq 2k, \end{cases}
\enqn
and adding it to $kq - \binom{k+1}{2}$ we obtain the wanted bound.
\end{proof}

\subsection{Bounds from positivity of Chern Classes}

The second method to improve Corollaries \ref{cor-k<d-1} and \ref{cor-k<d-2} uses the fact that the last Grassmannian BGG sheaves are globally generated. Although it can be used with any of the complexes $C_r^j$, we will focus on the case $C_3^0$, since it leads to more inequalitites involving $h=h^{2,0}$ and $q$ which we can compare with the previous ones. This approach generalizes some parts of \cite{LP} and \cite{Lom}.

Consider thus the complex $C_3^0$ over the Grasmannian $\mG_k$ for some $k$,
\begin{multline} \label{Grass-C30}
C_3^0: \quad 0 \longra \Sym^3 S \longra \Sym^2 S \otimes H^0\left(X,\Omega_X^1\right) \longra \\
\longra S \otimes H^0\left(X,\Omega_X^2\right) \longra H^0\left(X,\Omega_X^3\right) \otimes \caO_{\mG} \longra \caG=\caF_{3,3}^0 \longra 0,
\end{multline}
and assume that it is exact as a sequence of {\em sheaves} on $\mG$. As in the previous discussion, we do not know of better (geometric) hypothesis to be put directly on the variety $X$ and guaranteeing the exactness of (\ref{Grass-C30}).

Since $\caG$ is generated by global sections (it is a quotient of a trivial bundle), all its Chern classes must be represented by effective cycles, and this gives some inequalities involving $h,q$ and $k$ (the rank of $S$).

Without using the global generation, one can truncate the complex after $S \otimes H^0\left(X,\Omega_X^2\right)$ and use that the cokernel must have non-negative rank. This implies
\beqn 
h^{2,0} \geq \frac{1}{k}\left(q\binom{k+1}{2} - \binom{k+2}{3}\right) = \frac{k+1}{2}q-\frac{(k+2)(k+1)}{6},
\enqn
which is not better than $h \geq kq-\binom{k+1}{2}$ (the one obtained from the exactness of some $C_{2,W}^0$).

In order to use the global generation, we compute the lower terms of
\beqn
c\left(\caG\right) = \frac{c\left(\Sym^2 S\right)^q}{c\left(S\right)^hc\left(\Sym^3 S\right)} = \sum_{\lambda} g_{\lambda} \sigma_{\lambda},
\enqn
where $g_{\lambda} \in \mQ[h,q,k]$. Then, the family of inequalities we want to describe as explicitly as possible is $\left\{g_{\lambda} \geq 0\right\}$.

\subsubsection{Inequality from $c_1(\caG)\geq 0$}

From
\beqn
c_1\left(\caG\right) = q c_1\left(\Sym^2 S\right) + h c_1\left(Q\right) - c_1\left(\Sym^3 S\right) = \left(h-q\left(k+1\right)+\binom{k+2}{2}\right) \sigma_1
\enqn
we obtain the inequality

\beq \label{ineq_c1}
h \geq q(k+1)-\binom{k+2}{2}.
\enq

Note that this inequality is the same that we would have obtained from the exactness of $C^0_{2,W}$ for some $W$ of dimension $k+1$.

\subsubsection{Inequality from $c_2(\caG)\geq 0$}

After computing
\beqn
c_2\left(\caG\right) = g_2\sigma_2 + g_{1,1}\sigma_{1,1}
\enqn
we obtain
\begin{multline*} 
g_2 = \frac{1}{2}h^2 - \left(q(k+1)-\binom{k+2}{2}-\frac{1}{2}\right)h + \\
 + \left(\binom{q}{2}(k+1)^2-\frac{1}{2}q\left(k+2\right)\left(k^2+k+2\right)+\frac{1}{8}\left(k+3\right)\left(k+2\right)\left(k^2+k+4\right)\right)
\end{multline*}
and
\beqn 
g_{1,1} = \frac{1}{2}h^2 - \left(q(k+1)-\binom{k+2}{2}+\frac{1}{2}\right)h + \left(\binom{q}{2}(k+1)^2-qk\binom{k+2}{2}+3\binom{k+3}{4}\right).
\enqn

Considering $g_2$ and $g_{1,1}$ as quadratic polynomials in $h$, we can compute their roots formally, which are (for $g_2$ and $g_{1,1}$ respectively)
\beqn
\alpha_{\pm}=\left(q(k+1)-\binom{k+2}{2}-\frac{1}{2}\right)\pm\frac{1}{2}\sqrt{8(q-k)-15}
\enqn
and
\beqn 
\beta_{\pm}=\left(q(k+1)-\binom{k+2}{2}+\frac{1}{2}\right)\pm\frac{1}{2}.
\enqn
First of all, note that $\beta_{\pm}$ are consecutive integers, so $g_{1,1} \geq 0$ holds for all integers $h,k,q$ and it does not give any bound at all. Secondly, the roots $\alpha_{\pm}$ are not defined if $8(q-k)-15 < 0$, which is equivalent to $k \geq q-1$ (both $q$ and $k$ are integers). Therefore, for $k = q-1,q$ we again do not obtain any new bound. Assuming $k \leq q-2$, $g_2\geq0$ implies that either $h \geq \alpha_+$ or $h \leq \alpha_-$. But since $\alpha_- < q(k+1)-\binom{k+2}{2}$ and we already know that $h \geq q(k+1)-\binom{k+2}{2}$ (inequality (\ref{ineq_c1})), the option $h \leq \alpha_-$ is impossible, and we only obtain the following

\begin{pro}
If $X$ is an irregular variety and $k \leq q\left(X\right)-2$ is such that (\ref{Grass-C30}) is an exact sequence of sheaves on $\mG_k$, then
\beq \label{final_ineq_c2}
h^{2,0}\left(X\right) \geq q(k+1)-\binom{k+2}{2}+\frac{1}{2}\left(\sqrt{8q-(8k+15)}-1\right).
\enq
\end{pro}

\begin{rmk}
In the case $k=1$, the inequality (\ref{final_ineq_c2}) coincides with the results of Lombardi \cite{Lom} for threefolds.
\end{rmk}


\bibliographystyle{abbrv}

\end{document}